\documentclass[twoside]{report}
\usepackage{lj-igpl2}

\usepackage{amsmath}
\usepackage{amscd}
\usepackage{amssymb}
\usepackage{amsfonts}
\usepackage{euscript}
\usepackage{oldgerm}
\usepackage{calligra}
\usepackage{mathrsfs}
\usepackage{hyperref}
\usepackage{url}
\usepackage{wasysym}
\usepackage{pifont}

\Title[Non-First-Order-izability of Yablo's Paradox / {\sc Saeed Salehi}]{Sometime a Paradox, Now Proof: Non-First-Order-izability of Yablo's Paradox}

\ShortAuthor{S. Salehi}

\LongAuthor{
\author{\sc Saeed Salehi}
\address{Research Institute for Fundamental Sciences (RIFS), University of Tabriz, P.O.Box~51666-16471,
Tabriz, Iran. \\
School of Mathematics, Institute for Research in Fundamental Sciences, P.O.Box~19395--5746, Tehran, Iran. \quad\!\quad\!\quad
E-mail:~root@SaeedSalehi.ir}}

\begin{document}
\renewcommand{\xqed}{\hfill\ding{113}}

\begin{paper}

\begin{abstract}
Paradoxes are interesting puzzles in philosophy and mathematics,
and they could be even more fascinating when turned into proofs and theorems.
For example,  Liar's paradox can be translated into a propositional tautology, and  Barber's paradox turns into a  first-order  tautology.
Russell's paradox, which collapsed Frege's foundational framework,
is now a classical theorem in set theory, implying that no set of all sets can exist.
Paradoxes can be used in proofs of  some other theorems;   Liar's paradox has been used  in the classical proof of  Tarski's theorem on the undefinability of truth in sufficiently rich languages. This paradox (and also Richard's paradox) appears
implicitly in  G\"{o}del's proof of his celebrated first incompleteness theorem.
In this paper, we study  Yablo's paradox  from  the viewpoint of first  and  second order logics.  We prove that  a formalization of Yablo's paradox (which is second-order in nature)  is  non-first-order-izable  in the sense of  George Boolos (1984).

\bigskip


\smallskip

\end{abstract}

\bigskip

\begin{flushright}{\small
\textsl{\textbf{This was sometime a paradox, but now the time gives it proof.}}
\\[-0.25em]
--- {\it William Shakespeare} (\textsf{\footnotesize Hamlet,  Act  3, Scene 1}).
}\\
\end{flushright}

\section{Introduction}\label{sec:intro}
If mathematicians and philosophers have come to the conclusion that some (if not almost all) of the paradoxes cannot be (re)solved, or as Priest \cite[p.~160]{priest} puts it ``the programme of solving the paradoxes is doomed to failure'', they have learned to live (and also get along) with the paradoxes; of course, as long as the paradoxes do not crumble the foundations of our logical systems. Paradoxes have proved to be more than puzzles or destructive contradictions; indeed they have been used in proofs of some fundamental mathematico-logical theorems. Let us take the most well-known, and perhaps the oldest, paradox: Liar's paradox. When translated into the language of logic, this paradox seems to claim the existence of a sentence $\lambda$ such that $\lambda\!\longleftrightarrow\!\neg\lambda$ holds. Now, Liar's paradox can turn into a propositional tautology: $\neg(\textsl{\textsf{p}}\!\longleftrightarrow\!
\neg\textsl{\textsf{p}})$.
In fact, when trying to convince oneself
 about the truth
 of $\neg(\textsl{\textsf{p}}\!\longleftrightarrow\!
 \neg\textsl{\textsf{p}})$,
 one can see that  the supposed argument is not that much  different from the argument of Liar's paradox. One can clearly see that the paradox becomes a (semantic) proof for that tautology; hence the title of the article (which uses the above epigram of Shakespeare).

  Let us take a second example; Russell's paradox. If there existed a set $\textsl{\textsf{r}}$ such that $\forall x (x\!\in\!\textsl{\textsf{r}}\!\longleftrightarrow\!
  x\!\not\in\!x)$, then we would have a contradiction (the sentence $\textsl{\textsf{r}}\!\in\!\textsl{\textsf{r}}
\!\longleftrightarrow\!\textsl{\textsf{r}}\!\not
\in\!\textsl{\textsf{r}}$ which results from substituting $x$ with $\textsl{\textsf{r} }$). So, the sentence  $\neg\exists y\forall x(x\!\in\!y\!\leftrightarrow\!x\!\not\in\!x)$ is a theorem in the theory of sets, whose proof is nothing more than the argument of Russell's paradox. Going deeper into the proof (or the paradox), one can see that no real set-theoretic properties of the membership relation ($\in$) is used. That is to say that for an arbitrary binary relation $\textsl{\textsf{s}}$, the sentence $\neg\exists y\forall x[\textsl{\textsf{s}}(y,x)\!\longleftrightarrow\!\neg \textsl{\textsf{s}}(x,x)]$ is a first-order logical tautology (see \cite[Exercise~12, p.~76]{dalen}). Now, if we interpret the predicate  $\textsl{\textsf{s}}(y,x)$ as ``$y$ shaves $x$'', then we get Barber's paradox (due to Russell again). More generally, for any formula $\varphi(x,y)$ with the only free variables $x$ and $y$, the sentence $\neg\exists y\forall x[\varphi(x,y)\!\longleftrightarrow\!\neg\varphi(x,x)]$ is a first-order logical tautology, whose semantic proof is very similar to the argument of Russell's or Barber's paradox. In a similar way,  $\neg \exists X^{(2)}\exists y\forall x[X^{(2)}(x,y)\!\longleftrightarrow\!\neg X^{(2)}(x,x)]$ is a second-order  tautology.

In this paper, we are mainly interested in Yablo's paradox \cite{yablo85,yablo93}\footnote{A closely related paradox is Visser's~\cite{visser} which we do not study here.}; several papers (that we do not cite all of them here) and one book \cite{cook} have been written on different aspects of this paradox. Yablo's paradox says that if there existed a sequence of sentences $\{Y_n\}_{n\in\mathbb{N}}$ with the property  that for all $n\!\in\!\mathbb{N}$, the sentence  $Y_n$ is true if and only if  $Y_k$  is untrue for every $k\!>\!n$,
then we would have a contradiction, since none of those sentences can have a truth value (the sentences $Y_n$ would be neither true nor false). This paradox is humbly called by Yablo himself, the $\omega$-Liar paradox.
The paradoxicality of the sequence $\{Y_n\}_{n\in\mathbb{N}}$ of sentences with the above property follows form the observation that if $Y_m$ is true, for some $m$, then $Y_{m+1}$, and also all $Y_k$'s, for $k\!>\!m\!+\!1$, should be untrue. So, by the falsity of $Y_{m+1}$, there should exist some $j\!>\!m\!+\!1$ such that $Y_j$ is true; a contradiction. Whence, all $Y_m$'s should be untrue, and so $Y_0$ must be true; another contradiction!

\section{Yablo's Paradox in Seconc-Order Logic}\label{sec:2nd}
For formalizing Yablo's paradox in a (first-order or second-order)  language, we first abstract away even the order relation, that appears in the paradox, and  replace  it with an arbitrary binary relation symbol $R$; see \cite{hz17} for a non-arithmetical formulation of Yablo's paradox. Let us  take $\mathcal{Y}_1$ to be the first-order scheme $$\neg\forall x\big(\varphi(x)\longleftrightarrow\forall y [xRy\rightarrow\neg\varphi(y)]\big),$$
where $\varphi(x)$ is an arbitrary first-order formula with the only free variable $x$. Here, the sentences $Y_n$ are represented by $\varphi(n)$, and the quantifiers  of the form $\forall k\!>\!n\cdots$ are represented by $\forall k (kRn\!\rightarrow\!\cdots)$.

 \begin{definition}[$\mathcal{Y}$: Yablo's Paradox in Second Order Logic]\label{def:yablo12}
{\em Let $\mathcal{Y}$ be the following second-order sentence:
$$\neg\exists X^{(1)}\forall x\big(X^{(1)}(x)\longleftrightarrow\forall y[xRy\rightarrow\neg X^{(1)}(y)]\big),$$
where $R$ is a fixed binary relation symbol.}
\hfill\ding{71}\end{definition}
Some sufficient conditions for proving ($\mathcal{Y}_1$ and) $\mathcal{Y}$ are
\begin{itemize}\itemindent=2em
\item[$(A_1)$:]  $\forall x\exists y (xRy)$, \; and \quad
$(A_2)$: $\forall x,y,z\,(xRy\wedge yRz\rightarrow xRz)$.
\end{itemize}
That is to say that  ($A_1\!\wedge\!A_2\rightarrow  \mathcal{Y}_1$ is a first-order tautology, and)   $A_1\!\wedge\!A_2\rightarrow  \mathcal{Y}$ is a second-order tautology;  see \cite{ket05,hz17}.
None of these conditions are necessary for $\mathcal{Y}$; for example in the directed graph $\langle D;R\rangle$ with $D\!=\!\{a,b,c\}$ and $R\!=\!\{(a,b),(a,c),(c,c)\}$, we have $\mathcal{Y}$ and $A_2$ but not $A_1$. Also, in the directed graph $\langle D;R\rangle$ with $D\!=\!\{a,b,c\}$ and $R\!=\!\{(a,b),(b,c),(c,c)\}$, we have $\mathcal{Y}$ and $A_1$ but not $A_2$.

As a matter of fact, some weaker conditions than $A_1\wedge A_2$ can also prove ($\mathcal{Y}_1$ and) $\mathcal{Y}$. For example, the sentence
\begin{itemize}\itemindent=2em
\item[$(A)$:] $\forall x\exists y (xRy\wedge\forall z[yRz\rightarrow xRz])$,
\end{itemize}
suffices (see Theorem~\ref{thm:thetas} below). To see that $A$ is really weaker than $A_1\wedge A_2$, consider $\langle D;R\rangle$ with $D\!=\!\{a,b,c,d\}$ and $R\!=\!\{(a,b),(b,c),(a,d),(b,d),(c,d),(d,d)\}$. Then $(D;R)$ does not satisfy $A_2$, since it is not transitive (we have $aRbRc$ but $\neg aRc$), while it satisfies $A$, since for any $x\!\in\!D$ we have $xRd\wedge\forall z[dRz\rightarrow xRz]$. Even some weaker conditions than $A$ can prove ($\mathcal{Y}_1$ and) $\mathcal{Y}$.
 \begin{definition}[Some Sufficient Conditions for Proving $\mathcal{Y}$]\label{def:theta}
{\em
Let $\theta_0(x)$ be the formula $\exists y(xRy\wedge\forall z[yRz\!\rightarrow\!xRz])$.

\noindent
For any $n\!\in\!\mathbb{N}$, if $\theta_n(x)$ is defined, then let $\theta_{n+1}(x)\!=\!\exists y(xRy\wedge\forall z[yRz\!\rightarrow\!\theta_n(z)])$.
}\hfill\ding{71}\end{definition}

We now show that   $\{\forall x\,\theta_n(x)\}_{n\in\mathbb{N}}$ is a decreasing sequence of sentences (i.e., every sentence is stronger than its successor, in the sense that the sentence implies its successor but not vice versa) that all imply ($\mathcal{Y}_1$ and) $\mathcal{Y}$. Note that $A=\forall x\,\theta_0(x)$.

\begin{theorem}[
$\forall x\,\theta_{0}(x)\not\,\dashv\vdash\cdots
\forall x\,\theta_{n}(x)\not\,\dashv\vdash\forall x\,\theta_{n+1}(x)\not\,\dashv\vdash\cdots
\not\,\dashv\vdash\mathcal{Y}$ ]
\label{thm:thetas}
{\em For every $n\!\in\!\mathbb{N}$, we have}

{\em \textup{(1)}\; $ \forall x\,\theta_n(x) \vDash \mathcal{Y}$\textup{;}
\;
\textup{(2)}\;    $ \forall x\,\theta_n(x) \vDash \forall x\, \theta_{n+1}(x)$\textup{;}
\;
\textup{(3)}\;   $ \forall x\,\theta_{n+1}(x) \nvDash \forall x\, \theta_{n}(x)$.}
\end{theorem}
\begin{proof}$\hspace{-1.4ex}:$

\noindent
(1): \textsl{By induction on $n$. For $n\!=\!0$, take a directed graph $\langle D;R\rangle$ and assume that $\forall x\,\theta_0(x)$ holds in it. If   a subset $X\!\subseteq\!D$  satisfies  $\forall x(x\!\in\!X\!\leftrightarrow\!\forall y[xRy\!\rightarrow\!y\!\not\in\!X])$, then for any $a\!\in\!D$,  the assumption $a\!\in\!X$ implies that there exists some $b\!\in\!D$ such that $aRb$   and $\forall z[bRz\!\rightarrow\!aRz]$. Now, by $b\!\not\in\!X$, there should exist some $c\!\in\!D$ such that $bRc$ and $c\!\not\in\!X$. Also, $aRc$ should hold, which is a contradiction with $a\!\in\!X$. Thus $X\!=\!\emptyset$. But then for any $a\!\in\!D$ there should exist some $b\!\in\!D$ with $aRb$ and $b\!\in\!X$, and so $X\!\neq\!\emptyset$; another contradiction. Thus, there exists no such  $X\!\subseteq\!D$; whence, $\forall x\,\theta_0(x) \vDash  \mathcal{Y}$.}

\textsl{Now, suppose that $ \forall x\,\theta_n(x) \vDash \mathcal{Y}$ holds. Take a directed graph $\langle D;R\rangle$ and assume that $\forall x\,\theta_{n+1}(x)$ holds in it. If for  some  $X\!\subseteq\!D$ we have  $\forall x(x\!\in\!X\!\leftrightarrow\!\forall y[xRy\!\rightarrow\!y\!\not\in\!X])$, then for any $a\!\in\!D$, there exists some $b\!\in\!D$ such that $aRb$, and we have $aRx$ for any $x$ in the set $D_b\!=\!\{z\!\in\!D\mid bRz\}$. Now, if $D_b\!\neq\!\emptyset$, then the directed graph $\langle D_b,R\!\cap\!D_b^2\rangle$ satisfies $\forall x\,\theta_n(x)$, and so, by the induction hypothesis,
    the set $X\!\cap\!D_b$ cannot exist. So, we necessarily  have $D_b\!=\!\emptyset$. Now, if   $a\!\in\!X$ holds, then we should have that $b\!\not\in\!X$ and so there should exists some $c\!\in\!D_b$ with $c\!\not\in\!X$; a contradiction. Thus,  $X\!=\!\emptyset$. Then, for any $a\!\in\!D$, since $a\!\not\in\!X$, there should exist some $b$ with $aRb$ and $b\!\in\!X$; another  contradiction. This shows that $ \forall x\,\theta_{n+1}(x) \vDash \mathcal{Y}$.}

\noindent
(2):
\textsl{Suppose that  $\forall x\,\theta_n(x)$ holds, and fix an $x$. There is some $y$ such that $xRy$; and for any $z$ with $yRz$ we  have  $\theta_n(z)$ by the assumption $\forall x\,\theta_n(x)$. So, $\forall x\,\theta_{n+1}(x)$ holds.}

\noindent
(3):
\textsl{Consider $\langle D;R\rangle$, with  $D\!=\!\{a_0,a_1,\cdots,a_{2n}\}$ and  $R\!=\!\{(a_i,a_{i+1})\mid 0\!\leqslant\!i\!<\!2n\}\cup\{(a_{2n},a_{2n})\}$.
In the directed graph $\langle D;R\rangle$,  obviously,  $\forall x\,\theta_{n+1}(x)$ holds, but $\forall x\,\theta_{n}(x)$ does not hold, since we have $a_{2n-2}Ra_{2n-1}Ra_{2n}$ but $\neg(a_{2n-2}Ra_{2n})$.}
\end{proof}

As a result, $\mathcal{Y}$ does not imply the sentence  $\forall x\,\theta_n(x)$, for any $n\!\in\!\mathbb{N}$.
In the next section, we show that no first-order sentence in the language of $\langle R\rangle$ is equivalent with the second-order sentence  $\mathcal{Y}$. So, neither the sentence $\mathcal{Y}$ nor its negation  $\neg\mathcal{Y}$ is first-order-izable (see \cite{boolos,boolos1}). Not only the second-order sentence $\mathcal{Y}$ is non-equivalent with any first-order sentence, but also it is non-equivalent with any first-order theory (which could contain infinitely many sentences).
Actually, $\neg\mathcal{Y}$ is equivalent with the existence of a kernel in a directed graph $\langle D;R\rangle $; see e.g. \cite{berge}. So, our result shows that the existence or non-existence of a kernel in a directed graph is not equivalent to any first-order sentence (in the language of directed graphs). Whence, Yablo's paradox, formalized as ($\mathcal{Y}_1$ or as) $\mathcal{Y}$ in Definition~\ref{def:yablo12}, does not turn by itself into (a first-order or) a second-order tautology, and some conditions should be put on $R$ to make it a theorem. This paradox can be nicely translated into some theorems in Linear Temporal Logic (see \cite{ks14,ks17}) or in Modal Logic (see \cite{fg16}).

\section{Non-first-orderizability of Yablo's Paradox}\label{sec:not1}
Consider the language $\langle\mathfrak{s}\rangle$, where  $\mathfrak{s}$ is a unary function symbol. A standard
structure  on this language is $\langle\mathbb{N};\mathfrak{s}\rangle$, where  $\mathfrak{s}$ is interpreted as the successor function: $\mathfrak{s}(n)\!=\!n\!+\!1$ for all natural numbers  $n\!\in\!\mathbb{N}$.

\begin{definition}[Theory of Successor,  and Kernel of a Directed Graph]\label{def:syablo}
{\em Let the theory $\mathcal{S}$ on the language  $\langle \mathfrak{s}\rangle$ consist of the following axiom:
$$\forall x,y\,[\mathfrak{s}(x)\!=\!\mathfrak{s}(y)\rightarrow x\!=\!y].$$
With any structure $\langle M;\mathfrak{s}\rangle$,  the directed graph  $\langle M;R\rangle$ is associated, where $R$ is defined by $xRy \iff y\!=\!\mathfrak{s}(x)$, for all $x,y\!\in\!M$.}

\noindent
{\em For a directed graph $\langle D;R\rangle$, a subset $K\!\subseteq\!D$ is called a {\em Kernel}, when it has the following property:  $\forall x\big(x\!\in\!K\leftrightarrow\forall y[xRy\!\rightarrow\!y\!\not\in\!K]\big)$.}

\noindent
{\em For a formula $\varphi$ over the language $\langle R\rangle$, let
$\varphi^{\mathfrak{s}}$ result from $\varphi$ by replacing each $uRv$ with $\mathfrak{s}(u)\!=\!v$ for variables $u,v$; so, $\varphi^{\mathfrak{s}}$ is a formula over the language $\langle \mathfrak{s}\rangle$. }
\hfill\ding{71}\end{definition}

So, $\neg\mathcal{Y}$ states the existence of a Kernel in a directed graph with relation $R$, and the $\mathfrak{s}$-translation of Yablo's paradox  $\mathcal{Y}^{\mathfrak{s}}$ is equivalent with the second-order sentence  $\neg\exists X^{(1)}\forall x\big[X^{(1)}(x)\longleftrightarrow\neg X^{(1)}\big(\mathfrak{s}(x)\big)\big]$. Any structure $\langle M;\mathfrak{s}\rangle$ which satisfies  $\mathcal{S}$ may contain some copies of $$\mathbb{N}\approx\{a_0,a_1,a_2,\cdots\}$$ with $a_{n+1}\!=\!\mathfrak{s}(a_n)$ for all $n\!\in\!\mathbb{N}$, such that there is no $a\!\in\!M$ with $\mathfrak{s}(a)\!=\!a_0$. It  may also have some copies of $$\mathbb{Z}\approx\{\cdots,a_{-2},a_{-1},a_0,a_1,a_2,\cdots\},$$  in which $a_{m+1}\!=\!\mathfrak{s}(a_m)$ for all $m\!\in\!\mathbb{Z}$. There could be also some finite cycles $$\mathbb{Z}_m\approx\{a,\mathfrak{s}(a),
\mathfrak{s}^2(a),\cdots,\mathfrak{s}^{m-1}(a)\} \textrm{\, with \,} \mathfrak{s}^m(a)\!=\!a,$$
for some $m\!>\!0$. Let us note that, by the axiom $\mathcal{S}$, no two copies  of $\mathbb{N}$ or $\mathbb{Z}$ or a finite cycle can intersect one another.  Indeed, these are all a model $\langle M;\mathfrak{s}\rangle$ of $\mathcal{S}$ can contain.

\begin{lemma}[Axiomatizability of $\neg\mathcal{Y}^{\mathfrak{s}}\!+\!\mathcal{S}$]\label{cor:axioms}
{\em The associated directed graph $\langle D;R\rangle$ of a model $\langle M;\mathfrak{s}\rangle$ of $\mathcal{S}$ satisfies $\neg\mathcal{Y}$, if and only if $\langle M;\mathfrak{s}\rangle$ has no odd cycles, if and only if the structure $\langle M;\mathfrak{s}\rangle$ also satisfies the axioms $\neg\exists x(\mathfrak{s}^{2n+1}(x)\!=\!x)$ for 
$n\!\in\!\mathbb{N}$.}
\end{lemma}
\begin{proof}$\hspace{-1.4ex}:$

\noindent
\textsl{The second equivalence is straightforward; so, we prove the first equivalence only.}

\noindent
\textsl{First, suppose that $\langle M;\mathfrak{s}\rangle$ has no odd cycles. Then let $K\subseteq M$ consist of the even natural and integer numbers (as the copies of $\mathbb{N}$ and $\mathbb{Z}$) of $M$ (if any), and the elements with even indices in the finite cycles of $M$; i.e., for a finite even cycle $\{a,\mathfrak{s}(a),
\mathfrak{s}^2(a),\cdots,\mathfrak{s}^{2m+2}(a)\!=\!a\}$ take $\{a,\mathfrak{s}^2(a),\mathfrak{s}^4(a),\cdots, \mathfrak{s}^{2m}(a)\}$. Then the set $K$ is a kernel of $\langle M;R\rangle$, since  an element of $M$ is in $K$, if and only if it is even indexed, if and only if its successor is odd indexed, if and only if its successor is not in $K$.
Thus, $\langle M;R\rangle$ satisfies $\neg\mathcal{Y}$. This would have not been possible if there were an odd cycle; i.e., an element $\alpha$ such that $\mathfrak{s}^{2m+1}(\alpha)\!=\!\alpha$ for some $m\!>\!0$, since $\alpha$ would had been odd and even indexed at the same time.}

\noindent
\textsl{Second, suppose that
the directed graph
$\langle M;R\rangle$ associated to
$\langle M;\mathfrak{s}\rangle$ has a kernel $K$,  and also  (for the sake of a contradiction) that $\langle M;\mathfrak{s}\rangle$ has an odd  cycle such as  $\{a,\mathfrak{s}(a),\mathfrak{s}^2(a),\cdots,
\mathfrak{s}^{2m+1}(a)\!=\!a\}$,
for some $m\!>\!0$. Then, if $a\!\in\!K$, then $\mathfrak{s}(a)\!\not\in\!K$, then $\cdots$ $\mathfrak{s}^{2m}(a)\!\in\!K$, and so $a\!=\!\mathfrak{s}^{2m+1}(a)\!\not\in\!K$, a contradiction. Also, if $a\!\not\in\!K$, then $\mathfrak{s}(a)\!\in\!K$, then $\cdots$ $\mathfrak{s}^{2m}(a)\!\not\in\!K$, and so $a\!=\!\mathfrak{s}^{2m+1}(a)\!\in\!K$, a contradiction again. Therefore, if $\langle M;R\rangle$ has a kernel, then $\langle M;\mathfrak{s}\rangle$ can have no odd cycle.}
\end{proof}

\begin{theorem}[Non-First-Order-izability of $\mathcal{Y}$ and $\neg\mathcal{Y}$]\label{thm:nfo}
{\em The second-order sentence $\mathcal{Y}$ is not equivalent with any first-order sentence.}
\end{theorem}
\begin{proof}$\hspace{-1.4ex}:$

\noindent
\textsl{If there were a first-order sentence in the language $\langle R\rangle$ equivalent to $\mathcal{Y}$, then by Lemma~\ref{cor:axioms},  $\mathcal{S}'\!=\!\mathcal{S}\!\cup\!\{\neg\exists x(\mathfrak{s}^{2n+1}(x)\!=\!x)\mid n\!\in\!\mathbb{N}\}$ would be finitely axiomatizable (see \cite[Lemma 4.2.9]{dalen}). But this is not true, since for any finite subset of this theory, there exists a structure  that satisfies that finite sub-theory but is not a model of the whole  theory $\mathcal{S}'$: it suffices to take a sufficiently large odd cycle.}
\end{proof}

Thus, there can exist no  first-order sentence $\eta$ such that the second-order sentence $\eta\!\leftrightarrow\!\mathcal{Y}$  is a logical tautology. As a result, the proposed  formalization $\mathcal{Y}$ of  Yablo's paradox in Definition~\ref{def:yablo12}, being second-order in nature, is not (equivalent with any) first-order (sentence). We end the paper with a stronger result: there  cannot exist any  first-order theory that is equivalent with $\mathcal{Y}$. So, Yablo's paradox is not even infinitely first-order (i.e., it is non-equivalent even with any theory that consists of an infinite set of first-order sentences).

\begin{theorem}[Non-Equivalence of $\mathcal{Y}$ With First-Order Theories]\label{thm:notheory}
{\em The second-order sentence $\mathcal{Y}$ is not equivalent with any first-order theory.}
\end{theorem}
\begin{proof}$\hspace{-1.4ex}:$

\noindent
\textsl{By Lemma~\ref{cor:axioms}, the theory $\neg\mathcal{Y}^{\mathfrak{s}}+\mathcal{S}$ is axiomatizable over $\langle\mathfrak{s}\rangle$; if $\mathcal{Y}^{\mathfrak{s}}$ were axiomatizable, then by \cite[Lemma~4.2.10]{dalen} the theory $\neg\mathcal{Y}^{\mathfrak{s}}\!+\!\mathcal{S}$ would be finitely axiomatizable. But this contradicts Theorem~\ref{thm:nfo}. So, $\mathcal{Y}^{\mathfrak{s}}$ is not axiomatizable, hence $\mathcal{Y}$ is not equivalent with any first-order theory.}
\end{proof}

We conjecture that the second-order sentence $\neg\mathcal{Y}$, too, is non-equivalent with all the  first-order theories. This does not concern the main topic of this article, since the sentence $\neg\mathcal{Y}$ does not express Yablo's paradox, while  $\mathcal{Y}$ does that.

\subsubsection*{Acknowledgements} This research was partially  supported by a grant from  $\mathbb{IPM}$ ($\mathcal{N}^{\underline{\rm o}}$~$95030033$).  I warmly thank Kaave Lajevardi for drawing my attention to the Shakespearean epigram, that is quoted in this paper.


\end{paper}
\end{document}